\documentclass[12pt,a4paper]{article}
\usepackage{latexsym}
\usepackage{amsmath}
\usepackage{graphicx}
\usepackage{amsfonts}
\usepackage{amssymb}
\usepackage{amscd}
\usepackage{amsthm}
\usepackage[all]{xy}

\date{}

\begin{document}
\newtheorem{theorem}{Theorem}[section]
\newtheorem{mtheorem}[theorem]{Main Theorem}
\newtheorem{bbtheo}[theorem]{The $\bar\bschi$-Black Box}
\newtheorem{observation}[theorem]{Observation}
\newtheorem{proposition}[theorem]{Proposition}
\newtheorem{fprop}[theorem]{Freeness-Proposition}
\newtheorem{lemma}[theorem]{Lemma}
\newtheorem{testlemma}[theorem]{Test Lemma}
\newtheorem{mlemma}[theorem]{Main Lemma}
\newtheorem{note}[theorem]{}
\newtheorem{steplemma}[theorem]{Step Lemma}
\newtheorem{corollary}[theorem]{Corollary}
\newtheorem{notation}[theorem]{Notation}

\renewcommand{\labelenumi}{(\roman{enumi})}
\newcommand{\dach}[1]{\hat{\vphantom{#1}}}
\newcommand{\cp}{\widehat}
\newcommand{\dsum}{\bigoplus}

\newcommand{\pure}{\subseteq_\ast}

\def\Pf{\smallskip\goodbreak{\sl Proof. }}
\def\Fp{\vadjust{}\penalty200 \hfill
\lower.3333ex\hbox{\vbox{\hrule\hbox{\vrule\phantom{\vrule height
6.83333pt depth 1.94444pt width 8.77777pt}\vrule}\hrule}}
\ifmmode\let\next\relax\else\let\next\par\fi \next}

\def\Fin{\mathop{\rm Fin}\nolimits}
\def\br{\mathop{\rm br}\nolimits}
\def\fin{\mathop{\rm fin}\nolimits}
\def\Ann{\mathop{\rm Ann}\nolimits}
\def\Aut{\mathop{\rm Aut}\nolimits}
\def\End{\mathop{\rm End}\nolimits}
\def\bfb{\mathop{\rm\bf b}\nolimits}
\def\bfi{\mathop{\rm\bf i}\nolimits}
\def\bfj{\mathop{\rm\bf j}\nolimits}
\def\bM{\mathop{\rm\bf M}\nolimits}
\def\df{{\rm df}}
\def\bfk{\mathop{\rm\bf k}\nolimits}
\def\bEnd{\mathop{\rm\bf End}\nolimits}
\def\iso{\mathop{\rm Iso}\nolimits}
\def\id{\mathop{\rm id}\nolimits}
\def\Ext{\mathop{\rm Ext}\nolimits}
\def\Ines{\mathop{\rm Ines}\nolimits}
\def\Hom{\mathop{\rm Hom}\nolimits}
\def\Mon{\mathop{\rm Mon}\nolimits}
\def\bHom{\mathop{\rm\bf Hom}\nolimits}
\def\Rk{ R_\k-\mathop{\bf Mod}}
\def\Rn{ R_n-\mathop{\bf Mod}}
\def\map{\mathop{\rm map}\nolimits}
\def\ac{\mathop{\rm active }\nolimits}
\def\cd{\mathop{\rm cd}\nolimits}
\def\cf{\mathop{\rm cf}\nolimits}
\def\top{\mathop{\rm top}\nolimits}
\def\suc{\mathop{\rm suc}\nolimits}
\def\Ker{\mathop{\rm Ker}\nolimits}
\def\Bext{\mathop{\rm Bext}\nolimits}
\def\Br{\mathop{\rm Br}\nolimits}
\def\dom{\mathop{\rm Dom}\nolimits}
\def\min{\mathop{\rm min}\nolimits}
\def\im{\mathop{\rm Im}\nolimits}
\def\max{\mathop{\rm max}\nolimits}
\def\rk{\mathop{\rm rk}}
\def\Diam{\diamondsuit}
\def\Z{{\mathbb Z}}
\def\Q{{\mathbb Q}}
\def\N{{\mathbb N}}
\def\t{{\mathfrak t}}

\def\id{\mathop{\rm id}}

\makeatletter
\let\c@equation\c@theorem
\makeatother
\numberwithin{equation}{section}

\theoremstyle{definition}
\newtheorem{definition}[theorem]{Definition}
\newtheorem{example}[theorem]{Example}
\newtheorem{remark}[theorem]{Remark}


\def\bQ{{\bf Q}}
\def\bF{{\bf F}}
\def\bX{{\bf X}}
\def\bY{{\bf Y}}
\def\bHom{{\bf Hom}}
\def\bEnd{{\bf End}}
\def\bS{{\mathbb S}}
\def\AA{{\cal A}}
\def\BB{{\cal B}}
\def\CC{{\cal C}}
\def\DD{{\cal D}}
\def\TT{{\cal T}}
\def\FF{{\cal F}}
\def\GG{{\cal G}}
\def\PP{{\cal P}}
\def\SS{{\cal S}}
\def\R{{\cal R}}
\def\YY{{\cal Y}}
\def\fS{{\mathfrak S}}
\def\fH{{\mathfrak H}}
\def\fU{{\mathfrak U}}
\def\fW{{\mathfrak W}}
\def\fK{{\mathfrak K}}
\def\PT{{\mathfrak{PT}}}
\def\T{{\mathfrak{T}}}
\def\fX{{\mathfrak X}}
\def\fP{{\mathfrak P}}
\def\X{{\mathfrak X}}
\def\Y{{\mathfrak Y}}
\def\F{{\mathfrak F}}
\def\C{{\mathfrak C}}
\def\B{{\mathfrak B}}
\def\J{{\mathfrak J}}
\def\fN{{\mathfrak N}}
\def\fM{{\mathfrak M}}
\def\Fk{{\F_\k}}
\def\bar{\overline }
\def\Bbar{\bar B}
\def\Cbar{\bar C}
\def\Pbar{\bar P}
\def\etabar{{\bar \eta}}
\def\Tbar{\bar T}
\def\fbar{\bar f}
\def\nubar{{\bar \nu}}
\def\ubar{{\bar u}}
\def\vabar{{\bar \va}}
\def\vahat{{\hat \va}}
\def\psibar{{\bar \psi}}
\def\rhobar{\bar \rho}
\def\Abar{\bar A}
\def\a{\alpha}
\def\b{\beta}
\def\g{\gamma}
\def\w{\omega}
\def\e{\varepsilon}
\def\om{\omega}
\def\va{\varphi}
\def\k{\kappa}
\def\m{\mu}
\def\n{\nu}
\def\r{\rho}
\def\f{\phi}
\def\hv{\widehat\v}
\def\hF{\widehat F}
\def\v{\varphi}
\def\s{\sigma}
\def\l{\lambda}
\def\lo{\lambda^{\aln}}
\def\d{\delta}
\def\z{\zeta}
\def\th{\theta}
\def\ale{{\aleph_1}}
\def\aln{{\aleph_0}}
\def\al-n{{\aleph_n}}
\def\Cont{2^{\aln}}
\def\nld{{}^{ n \downarrow }\l}
\def\n+1d{{}^{ n+1 \downarrow }\l}
\def\hsupp#1{[[\,#1\,]]}
\def\size#1{\left|\,#1\,\right|}
\def\Binfhat{\widehat {B_{\infty}}}
\def\Zhat{\widehat \Z}
\def\Mhat{\widehat M}
\def\Rhat{\widehat R}
\def\Phat{\widehat P}
\def\Fhat{\widehat F}
\def\fhat{\widehat f}
\def\Ahat{\widehat A}
\def\Chat{\widehat C}
\def\Ghat{\widehat G}
\def\Bhat{\widehat B}
\def\Btilde{\widetilde B}
\def\Ftilde{\widetilde F}
\def\restl{\mathop{\upharpoonleft}}
\def\restr{\mathop{\upharpoonright}}
\def\to{\rightarrow}
\def\arr{\longrightarrow}
\def\bschi{\boldsymbol\chi}
\newcommand{\norm}[1]{\text{$\parallel\! #1 \!\parallel$}}
\newcommand{\supp}[1]{\text{$\left[ \, #1\, \right]$}}
\def\set#1{\left\{\,#1\,\right\}}
\newcommand{\mb}{\mathbf}
\newcommand{\wt}{\widetilde}
\newcommand{\card}[1]{\mbox{$\left| #1 \right|$}}
\newcommand{\union}{\bigcup}
\newcommand{\inters}{\bigcap}
\def\Proof{{\sl Proof.}\quad}
\def\fine{\ \black\vskip.4truecm}
\def\black{\ {\hbox{\vrule width 4pt height 4pt depth
0pt}}}
\def\fine{\ \black\vskip.4truecm}
\long\def\alert#1{\smallskip\line{\hskip\parindent\vrule%
\vbox{\advance\hsize-2\parindent\hrule\smallskip\parindent.4\parindent%
\narrower\noindent#1\smallskip\hrule}\vrule\hfill}\smallskip}

\def\map{\mathop{\rm map}\nolimits}
\def\ker{\mathop{\rm ker}\nolimits}
\def\im{\mathop{\rm im}}
\def\precdot{\mathop{\prec\!\!\!\cdot}\nolimits}
\def\colim{\mathop{\rm colim}}
\def\lim{\mathop{\rm lim}}
\def\ohom{\mathop{\rm \overline{Hom}}\nolimits}
\def\Hom{\mathop{\rm Hom}\nolimits}
\def\Rep{\mathop{\rm Rep}\nolimits}
\def\Aut{\mathop{\rm Aut}}
\def\id{\mathop{\rm id}\nolimits}
\def\Ast{\mathop{*}}
\def\vert{\mathop{\textsf{vert}}\nolimits}
\def\edge{\mathop{\textsf{edge}}\nolimits}
\def\aa{\mathbb{A}}
\def\g{\mathop{\mathcal Graphs}\nolimits}
\def\a{\mathop{\mathcal Ab}\nolimits}

\renewcommand{\theenumi}{\alph{enumi}}

\title{\sc An axiomatic construction of an almost full embedding
  of the category of graphs  into the category of $R$-objects}

\footnotetext{This paper is supported by the project No. I-963-98.6/2007 of
the German-Israeli Foundation for Scientific Research \& Development.
}

\author{R\"udiger G\"obel and Adam J. Prze\'zdziecki}

\date{}

\maketitle

\begin{abstract}
We construct embeddings $G$ of the category of graphs into categories of $R$-modules over a commutative ring $R$ which are almost full in the sense that the maps induced by the functoriality of $G$
$$
  R[\Hom_{{\mathcal G}raphs}(X,Y)]\longrightarrow\Hom_R(GX, GY)
$$
are isomorphisms. The symbol $R[S]$ above denotes the free $R$-module with the basis $S$. This implies that, for any cotorsion-free ring $R$, the categories of $R$-modules are not less complicated than the category of graphs. A similar embedding of graphs into the category of vector spaces with four distinguished subspaces (over any field, e.g. $\mathbb{F}_2=\{0,1\}$) is obtained.
\end{abstract}

\section{Introduction}\label{sec1}
For all commutative, cotorsion-free rings $R$ we construct embeddings $G$ of the category of graphs into categories of $R$-modules which are almost full in the sense that the maps induced by the functoriality of $G$
\begin{equation}\label{equation-introduction}
  R[\Hom_{{\mathcal G}raphs}(X,Y)]\longrightarrow\Hom_R(GX, GY)
\end{equation}
are isomorphisms. The symbol $R[S]$ above denotes the free $R$-module with the basis $S$. This notation is chosen so as to correspond to the one used for group rings -- when the category has only one object $*$ and $\Hom(*,*)$ is the group, then $\mathbb{Z}[\Hom(*,*)]$ is the group ring. A similar embedding of graphs into the category of vector spaces with four distinguished subspaces (over any field, e.g. $\mathbb{F}_2=\{0,1\}$) is derived in Corollary~\ref{corollary-r4}. The rings act on modules from the right and the functions act on their arguments from the right. The graphs are always directed, without multiple edges.

The functor $G$ is obtained as a $\kappa$-directed colimit. We need the $\aleph_1$-directedness in order to (\ref{equation-introduction}) be an isomorphism. The higher directedness may be used to obtain additional properties of the values of $G$ like $\aleph_n$-freeness, see Corollary~\ref{corollary-stronger-embedding}.

Some parts of the construction depend on the ring $R$ while other parts do not. The ring independent part of the proof consists of Section~\ref{Axi}, where we formulate axioms, and of Section~\ref{Axiproof}, which derives the existence of the embedding from these axioms (Main Theorem \ref{axThm}). For technical reasons it is convenient to formulate the axioms in a more general setting of $R$-objects. The ring dependent part, Section \ref{section-applications}, consists on a verification of the axioms in several categories of $R$-objects.

The category of graphs is, up to set theoretic considerations, universal among the concrete categories, that is every other concrete category fully embeds into it \cite{trnkova-book}. In fact, most interesting concrete categories are accessible and therefore fully embed into $\g$ without any set theoretic difficulties \cite{adamek-rosicky}. This motivated a research aimed at finding which categories admit a subcategory isomorphic to $\g$. These include, for example, the category of semigroups -- Hedrl\'in and Lambek \cite{semigroups}, the category of integral domains -- Fried and Sichler \cite{integral}. Many interesting categories, though, have no such subcategories for trivial reasons like the existence of constant maps in the category of spaces, conjugation or addition in the category of groups or abelian groups. As a result this research was naturally extended into almost full embeddings. There exist full embeddings of $\g$, up to constant maps, into metric spaces -- Trnkov\'a \cite{trnkova-metric}, paracompact spaces -- Koubek \cite{paracompact}; up to null-homotopic maps -- into the unpointed homotopy category and up to trivial homomorphisms and conjugation in the target -- into the category of groups \cite{przezdziecki-groups}. An almost full embedding, in the sense adopted in this paper, into the category of abelian groups (that is $\mathbb{Z}$-modules) was recently constructed in \cite{P}, our work is a wide generalization of that result.

The term {\em almost full embedding} was introduced in the seventies of the past century by Koubek \cite{paracompact}, who gave a convenient name (and approach) to a concept introduced two years earlier by Trnkov\'a \cite{trnkova-metric}. Their definition stated that an embedding was almost full if it was full up to constant morphisms (the target category was assumed to be concrete). Since the ``Prague school'' used this term only when the target category was some subcategory of Topological Spaces -- it was clear that the ``metadefinition'' was: {\em an embedding is almost full if it is full up to ``obvious sacrifices'' imposed by the target category}. In this broader sense the term was used in \cite{przezdziecki-groups} and \cite{P}. In this paper the target category is $R$-Modules, possibly with some additional structure. Therefore the metadefinition specifies to: {\em an embedding into $R$-Mod is almost full if every morphism in the target is an $R$-linear combination of morphisms coming from the source}.

\section{The axioms}\label{Axi}

\def\rc{R\mathcal{C}}

Here we list the axioms which will be used in Section~\ref{Axiproof} to prove Main Theorem~\ref{axThm}. The list is divided into two groups: Axioms of type A describe a category of $R$-objects $\rc$. A similar one is called a pre-module $R$-category in \cite{GM_endomorphism}. Axioms of type B describe certain object $M$, which we assume exists in $\rc$.

{\em Convention.} Throughout the paper:
\begin{itemize}
  \item $R$ is a commutative ring with $1\neq 0$
  \item $\kappa$ is a fixed cardinal of uncountable cofinality
\end{itemize}

\noindent The description of $\rc$:

\begin{itemize}
  \item[A1] $\rc$ is an $R$-category: for $A,B\in\rc$ the $\Hom_R(A,B)$ have a natural structure of an $R$-module and the composition is bilinear.
  \item[A2] $\rc$ is preabelian (hence kernels and cokernels exist).
  \item[A3] Arbitrary direct sums (hence colimits) exist in $\rc$.
  \item[A4] $\rc$ is concrete in such a way that:
  \begin{itemize}
    \item[(i)] The underlying set functor preserves directed colimits.
    \item[(ii)] Monomorphisms in $\rc$ are one-to-one on the underlying sets.
  \end{itemize}
\end{itemize}

The categories $\rc$, satisfying the axioms A1 -- A4, which we consider in this paper are: the category of $R$-modules and the category $R_4$-{\bf Mod} of $R$-modules with $4$ distinguished submodules, see Definition~\ref{definition-r-4-mod}.

Let $\Gamma$ denote a full subcategory of the category of graphs containing one representative for each isomorphism class of graphs of cardinality less than $\kappa$. Let $A$ be a free $R$-module whose basis consists of the morphisms of $\Gamma$ and $1$. We equip $A$ with a structure of an $R$-algebra by declaring that the basis elements multiply by composition or, if not composable, their product is zero. This multiplication is extended to all of $A$ by $R$-linearity, the elements of $R$ commute with the elements of $\Gamma$. In particular morphisms of $\Gamma$ may be viewed as elements of the algebra $A$.

There exists an object $M$ in $\rc$ such that:
\begin{itemize}
  \item[B1] $A\cong\Hom_R(M,M)$, in particular we may view $M$ as a right $A$-object.
  \item[B2] $A\subseteq M$ as the $A$-objects, and for all $\varphi\in\Gamma$ we have $A\varphi=A\cap M\varphi$ ($\Gamma$-purity).
  \item[B3] If $\varphi:C\arr D$ a monomorphism in $\Gamma$, then the composition $M\id_C\subseteq M\overset{\varphi}{\longrightarrow}M$ is a monomorphism.
\end{itemize}

\begin{remark}\label{remark-faithful-restriction}
  Since $A$ has the identity element, Axioms B1 and B2 imply that two $R$-endomorphisms of $M$ coincide if their restrictions to $A$ coincide.
\end{remark}

\section{The axiomatic proof of the almost full embedding}\label{Axiproof}

This section proves Main Theorem~\ref{axThm}. It is an adaptation of \cite[Section 3]{P} to the axiomatic setting, introduced in Section~\ref{Axi} above. Following a convention adopted in most of the related literature (\cite{adamek-rosicky}, \cite{trnkova-book}, and others), a {\em graph} is a directed graph, that is a set $X$ endowed with a binary relation $E\subseteq X\times X$. A morphism of graphs is a function $f:X\to X'$ preserving the relation. The members of $X$ are called {\em vertices} and the members of $E$ are called {\em edges}.

\begin{remark}\label{remark-m-splits} Let $\id_X:X\arr X$ be the identity map. If $X$ is in $\Gamma$, then $\id_X$ is an idempotent of $A\cong\Hom_R(M,M)$. Since we work in an additive category with kernels, \cite[Proposition 18.5]{mitchell} implies that we have $M\cong M\id_X\oplus M(1-\id_X)$.
\end{remark}

\begin{definition} \label{definition-g}
\begin{itemize}
  \item[(a)]
  For $X$ in $\Gamma$ we define $GX=M\id_X$. If $\varphi:X\arr Y$ is a map in $\Gamma$, then $\varphi\id_Y=\varphi$, hence
  $M\id_X\varphi\subseteq M\id_Y$, and therefore $\varphi$ induces, via right multiplication, an $R$-homomorphism $GX\arr GY$. Thus $G$ is a functor from $\Gamma$ to the category $\rc$.
  \item[(b)]
  If $X$ is an arbitrary graph, then we define
  $$
    GX=\colim_{c\in\Gamma\downarrow X}Gc
  $$
  where $\Gamma\downarrow X$ is the category of maps $c:C\arr X$ with $C$ in $\Gamma$ and $Gc$ is defined as $GC$.
  We may view this as an extension of (a) since for $X$ in $\Gamma$ the category $\Gamma\downarrow X$ contains a terminal object $\id_X$.
\end{itemize}
\end{definition}

\begin{remark} \label{remark-colimit} For an arbitrary graph $X$ we have
  \begin{equation}\label{equation-gx-def}
    GX=\colim_{C\in[X]^{<\kappa}}GC.
  \end{equation}
  Where $[X]^{<\kappa}$ denotes the category of inclusions whose objects are subgraphs $C\subseteq X$ whose cardinality is less than $\kappa$.
  This is clear since $[X]^{<\kappa}$ is isomorphic to a cofinal subcategory of $\Gamma\downarrow X$.
  A map $f: X\arr Y$ induces, by taking images, a map $[X]^{<\kappa}\arr[Y]^{<\kappa}$,
  which in turn induces a map $Gf:GX\arr GY$ of colimits.
\end{remark}

\begin{remark} The main reason we need to include infinite subgraphs in $[X]^{<\kappa}$ is that we want this poset to be countably directed to apply Lemma~\ref{lemma-limits-zs} in the proof of Theorem~\ref{theorem-embedding-graphs}, hence the assumption that $\kappa>\aleph_0$. All statements of this section except Lemma~\ref{lemma-limits-zs} and Theorem~\ref{theorem-embedding-graphs} hold when we take $\kappa=\aleph_0$. Considering $\kappa>\aleph_1$ may be useful to obtain additional properties of the values of the functor $G$, see Corollary~\ref{corollary-stronger-embedding}.
\end{remark}

\begin{remark} \label{remark-gamma-x}
  Let $X$ be in $\Gamma$. Let $\Gamma_X$ denote the set of maps $\sigma:C\arr X$ in $\Gamma$ whose target is $X$. It is a subset of the morphisms of $\Gamma$ hence a subset of the basis of $A$ viewed as an $R$-module. In the decomposition $M\cong M\id_X\oplus M(1-\id_X)$, noted in Remark~\ref{remark-m-splits}, we have, by Axiom B2, $\Gamma_X\subseteq M\id_X\cong GX$, and the complement of $\Gamma_X$ in ${\rm Mor}\,\Gamma$ is naturally bijective to a subset of $M(1-\id_X)$.
\end{remark}

\begin{lemma}\label{lemma-injective}
  If $\varphi:X\arr Y$ is one-to-one, then so is $G\varphi:GX\arr GY$.
\end{lemma}
\begin{proof}
  The case when $\varphi$ is in $\Gamma$ follows immediately from Axiom B3 and the definition of $GX$.

  Since $[X]^{<\kappa}$ is a category of inclusions, the previous paragraph implies that (\ref{equation-gx-def}) is a directed colimit of inclusions and therefore, by Axiom A4, we have $GC\subseteq GX$ for every $C\subseteq X$ with $|C|<\kappa$. For any $x\neq y$ in $GX$, there exists a $C\subseteq X$, $|C|<\kappa$ such than $x$ and $y$ belong to $GC$. Then $GC$ is mapped isomorphically to $G(C\varphi)$ and, since analogously $G(C\varphi)\subseteq GY$, we obtain $(x)G\varphi\neq (y)G\varphi$.
\end{proof}

\begin{lemma}\label{lemma-right-multiplication}
  Every $R$-homomorphism $h:GX\arr GY$ with $X$ and $Y$ in $\Gamma$ can be uniquely represented as right multiplication by an $a\in A$. The element $a$ is of the form $a=\sum_{i\in I}\sigma_ik_i$, where $k_i$ are nonzero elements of $R$, $\sigma_i:X\arr Y$ are distinct maps in $\Gamma$, and $I$ is finite.
\end{lemma}

\begin{proof}
  Let $\alpha:M\arr M$ be the composition $M\stackrel{(-)\cdot\id_X}{-\!\!\!\longrightarrow}GX\stackrel{h}{\longrightarrow}GY\subseteq M$.
  Axiom B1 implies that $(x)\alpha=xa$ for all $x\in M$ and some $a$ in $A$.
  By the definition of $A$ we have a unique representation $a=\sum_{i\in I}\sigma_ik_i$ where $k_i\in R$, $\sigma_i:X_i\arr Y_i$ are distinct maps in $\Gamma$ (or the identity), and $I$ is finite.
  Remark~\ref{remark-faithful-restriction} allows us to restrict $\alpha$ to $\alpha':A\arr A$.
  Since $\alpha'=\id_X\alpha'\id_Y$ we see that $X_i=X$ and $Y_i=Y$ for $i\in I$, and the $\sigma_i$ above can not be equal to the identity of $A$.
\end{proof}

The following two lemmas are tautological thanks to Axioms B1 and B2.

\begin{lemma} \label{lemma-dense}
  For any $R$-homomorphism $h:GX\arr GY$ with $X$ and $Y$ in $\Gamma$, if $(\id_X)h=\{0\}$, then $h=0$.
\end{lemma}
\begin{proof}
   Lemma~\ref{lemma-right-multiplication} yields an
   $a=\sum_{i\in I}\sigma_ik_i$ such that $(x)h=xa$ where $x\in GX$, and $\sigma_i:X\arr Y$ are distinct maps in $\Gamma$.
   Then $(\id_X)h=\id_X\sum \sigma_ik_i=\sum \sigma_ik_i\in GY$.
   The assumption that $(\id_X)h=0$ and Axiom B2 imply $k_i=0$ for all $i$ and therefore $a=0$ and $h=0$.
\end{proof}

\begin{lemma}\label{lemma-factors}
  Let $X,Y,W$ be graphs in $\Gamma$, $\varphi:X\arr Y$ a monomorphism and $h:GW\arr GY$ an $R$-homomorphism.
  If $(\id_W)h\in(GX)G\varphi$, then $h$ factors through $(GX)G\varphi$, i.e. there exists an $\tilde{h}:GW\arr GX$ with $\tilde{h}(G\varphi)=h$.
\end{lemma}

\begin{remark}\label{remark-factors}
  From the proof we obtain $\tau_j:W\arr X$ in $\Gamma$ and $k_j\in R$ such that
  $h=\sum k_j\tau_j\varphi$ and $\tilde{h}=\sum k_j\tau_j$. The lemma can be illustrated by the following diagram
  $$
  \xymatrix{
    {\{\id_W\}} \ar[r]
        \ar@{}[d]|(0.44){\rotatebox{-90}{$\subseteq$}} &
      GX \ar[d]^{G\varphi} \\
    GW \ar[r]^h \ar@{-->}[ur]^{\tilde{h}} &
      GY
  }
  $$
\end{remark}

\begin{proof}
  There exists $u\in GX$ such that $(\id_W)h=(u)G\varphi$.
  Lemma~\ref{lemma-right-multiplication} implies that $(\id_W)h=\sum \sigma_ik_i\in GY$ for some $k_i\in R$ and maps $\sigma_i: W\arr Y$ in $\Gamma$.
  In particular we have $u\varphi=(u)G\varphi=\sum\sigma_ik_i\in GY\cap A$. The injectivity of $G\varphi$, which we have by Lemma~\ref{lemma-injective}, and Axiom B2 implies that $u\in GX\cap A$, thus $u$ is of the form $\sum\tau_jn_j$ for some $n_j\in R$ and maps $\tau_j$ in $\Gamma$. We obtain
  $$\sum_{i\in I} \sigma_ik_i=\sum_{j\in J}\tau_j\varphi  n_j$$
  and therefore, by $R$-linear independence, we have a bijection $t:I\arr J$ such that $k_i=n_{(i)t}$ and $\sigma_i=\tau_{(i)t}\varphi$. Since $\varphi$ is a monomorphism the $\tau_{(i)t}$'s are uniquely determined by $\sigma_i$'s and both triangles in Diagram \ref{remark-factors} commute.
\end{proof}

\begin{lemma}\label{lemma-factorization-id}
   Let $h:GX\arr GY$ be an $R$-homomorphism with $|X|<\kappa$. If $C\subseteq Y$ is a subgraph such that $|C|<\kappa$ and $(\id_X)h\in GC\subseteq GY$, then $h$ factors through $GC\subseteq GY$.
\end{lemma}

\begin{proof}
  Fix a subgraph $D$ such that $C\subseteq D\subseteq Y$ and $|D|<\kappa$. Let $D_*$ be the full graph with the same vertices as $D$ and all possible edges.  Thus there is some extension $\delta:Y\arr D_*$ of the inclusion $D\subseteq D_*$ to $Y$. We may put these into the following diagram.
  $$
    \xymatrix{
      {\{\id_X\}} \ar[r]
          \ar@{}[d]|{\rotatebox{-90}{$\subseteq$}} &
        GC \ar@{}[r]|{\rotatebox{0}{$\subseteq$}} &
        GD \ar@{}[r]|{\rotatebox{0}{$\subseteq$}}
          \ar[dr]|{\rotatebox{-45}{$\subseteq$}}   &  
        GY  \ar[d]^{G\delta} \\
      GX \ar[rrr]_{h'} \ar[rrru]_h \ar@{-->}[ur]^{\tilde{h'}} &&&
        GD_*
    }
  $$
  The top line, the left vertical inclusion and the $h$ are given by assumptions. The right vertical map $G\delta :GY\arr GD_*$ is induced by $\delta$. The homomorphism $h'$ is the composition of $h$ and this extension $G\delta$. Thus $((\id_X)h)G\delta\in(GD)G\delta$.

  Lemma~\ref{lemma-factors} applied to $h'$, $X$ as $W$ and the inclusion $C\subseteq D_*$ as the monomorphism $\varphi$, gives us the dashed homomorphism $\tilde{h'}$. The central trapezoid commutes. Thus $h'$ factors through $\tilde{h}'$. Lemma~\ref{lemma-dense} implies that $\tilde{h}'$ does not depend on the choice of $D$: it is uniquely determined by its value on $\id_X$. If $h\neq\tilde{h}'$, then for some $x\in GX$ we have $(x)h\neq(x)\tilde{h}'$. Axiom A4(i) allows us to choose $D$ so that $(x)h\in GD$ and we obtain a contradiction as $(x)h=(x)\tilde{h}'$ in $GD\subseteq GD_*$.
\end{proof}

We will need the following immediate consequence of this Lemma.

\begin{corollary}\label{corollary-factorization}
  If $h:GX\arr GY$ is an $R$-homomorphism with $\size{X} <\kappa$, then there exists a  subgraph $C\subseteq Y$ of size $\size{C} < \kappa$ such that $h$ factors through $GC\subseteq GY$.
\end{corollary}

Functoriality of $G$ gives us a natural $R$-homomorphism $$\gamma:R[\Hom_{\g}(X,Y)]\arr\Hom_R(GX,GY)\ (\sum r_\va \va\mapsto \sum r_\va G\va ),$$
where $\va\in \Hom_{\g}(X,Y)$ and $r_\va\in R$.

\begin{remark} \label{remark-colimits}
  Lemma~\ref{lemma-right-multiplication} implies
  that $\gamma$ is an isomorphism when both $X$ and $Y$ have size $<\kappa$. Axiom A4(i) and Corollary~\ref{corollary-factorization} imply that it is enough that $X$ has size $<\kappa$ since then
$$
  R[\Hom(X,Y)]\cong R[\Hom(X,\colim_{C\in[Y]}C)]\cong
  \colim_{C\in [Y]}R[\Hom(X,C)]
  \overset{\colim\gamma}{\underset{\cong}{-\!\!\!-\!\!\!-\!\!\!\longrightarrow}}
$$

$$
  \arr\colim_{C\in[Y]}\Hom_R(GX,GC)
  \overset{\cong}{\longrightarrow}\Hom_R(GX, \colim_{C\in[Y]}GC)\cong
  \Hom_R(GX,GY).
$$
The last arrow being an isomorphism is equivalent to Corollary~\ref{corollary-factorization}.
\end{remark}

\begin{lemma}
\label{lemma-limits-zs}
  Let $\{S_i\}_{i\in I}$ be a diagram of sets. Let
  $\lambda:R[\lim S_i]\arr\lim R[S_i]$ be defined by the universal property of limits. If $I$ is codirected, then $\lambda$ is one-to-one and if $I$ is countably codirected, then $\lambda$ is an isomorphism.
\end{lemma}
\begin{proof}
  This is \cite[Lemma 3.13]{P}, for convenience of the reader we repeat the short proof.

  Let $0\ne a=\sum k_ss\in R[S]$ be an element of a free $R$-module $R[S]$ with some basis $S$. Let $|a|$ denote the {\em support} of $a$, consisting of those $s$ for which $k_s\neq 0$. If $a\in R[\lim S_i]$ then, since $I$ is codirected, there exists an $i\in I$ such that $|a|$ projects injectively into $S_i$ and therefore $(a)\lambda\neq 0$.

  Let $a=(a_i)\in\lim R[S_i]$. If there exists a sequence $i_n$, $n\in\mathbb{N}$, such that the supremum of the cardinalities of $|a_{i_n}|$ is $\omega$ then, since $I$ is countably codirected, there exists $S_{i_0}$ which maps to all $S_{i_n}$; but then $|a_{i_0}|$ must be infinite, a contradiction. Let $i_m\in I$ be such that $|a_{i_m}|$ is largest possible. Then for each $i<i_m$ the map $S_i\arr S_{i_m}$ restricts to a bijection $|a_i|\arr|a_{i_m}|$, hence the inclusions $|a_i|\subseteq S_i$ lift in a coherent way to
  $\lim_{i<i_m} S_i$. Since the set $\{i\mid i<i_m\}$ is coinitial in $I$, we have $\lim_{i<i_m}S_i=\lim S_i$, and therefore $a$ is in the image of $\lambda$.
\end{proof}

\begin{mtheorem}\label{axThm} \label{theorem-embedding-graphs}
  If a category $\rc$ satisfies Axioms A1--A4 and contains an object that satisfies B1--B3, then there exists a functor $G$ from the category of graphs to $\rc$ which induces natural isomorphisms
  $$
  \gamma:R[\Hom_{\g}(X,Y)]\overset{\cong}{\longrightarrow}\Hom_R(GX,GY)
  $$
\end{mtheorem}

\begin{proof}
  We have a chain of isomorphisms
  $$
  R[\Hom(X,Y)]\cong R[\Hom(\colim_{C\in[X]^{<\kappa}}C,Y)]\cong
  R[\lim_{C\in[X]^{<\kappa}}\Hom(C,Y)]{\overset{\lambda}{\longrightarrow}}
  $$
  $$
  \arr\lim_{C\in[X]^{<\kappa}}R[\Hom(C,Y)]{\overset{\lim\gamma}{-\!\!\!\longrightarrow}}
  \lim_{C\in[X]^{<\kappa}}\Hom_R(GC,GY)\cong
  $$
  $$
  \cong\Hom_R(\colim_{C\in[X]^{<\kappa}}GC,GY)\cong\Hom_R(GX,GY).
  $$
  Lemma~\ref{lemma-limits-zs} implies that $\lambda$ is an isomorphism, Remark~\ref{remark-colimits} implies that $\lim\gamma$ is an isomorphism as the limit of isomorphisms, and the last isomorphism follows from the definition of $GX$.
\end{proof}

\begin{corollary}
  If $R$ is a cotorsion-free ring with $1\neq 0$, then there exists a functor $G$ from the category of graphs to the category of $R$-modules which induces natural isomorphisms
  $$
  \gamma:R[\Hom_{\g}(X,Y)]\overset{\cong}{\longrightarrow}\Hom_R(GX,GY)
  $$
\end{corollary}

\begin{proof}
  This follows from Corollary~\ref{corollary-result-black-box}.
  If $|R|<2^{\aleph_0}$, then we may use Corollary~\ref{corollary-result-corner} instead.
\end{proof}

\subsection{No countable representation exist}

Our construction of the functor $G$ starts with the category $\Gamma$ whose cardinality is at least the continuum. Its size is inherited by the values of the functor so that the cardinality of $GX$ is always at least the continuum, even when $X$ is empty. Unfortunately one can not avoid this.

Let $R$ be a countable, cotorsion-free ring. Then free $R$-modules $F$ are  slender but the product $R^\omega$ of $\omega$ copies of $R$ is not. In particular $F$ contains no submodules isomorphic to $R^\omega$.

\begin{lemma}\label{BS-group}
  If $A$ and $B$ are torsion-free countable $R$-modules and $H=\Hom_R(A,B)$ is uncountable, then $H$ contains a copy of $R^\omega$.
\end{lemma}

\begin{proof}
  Let $\{a_1,a_2,\ldots\}$ be a set of generators for $A$.
  Let $H_n=\{f\in H\mid (a_i)f=0 \mbox{ for } i=1,2,\ldots, n\}$.
  Clearly $\bigcap H_n=0$. If there exists an $n$ such that $H_n=0$, then each $f$ in $H$ is uniquely determined by its values on $a_1,a_2,\ldots,a_n$ hence $H$ is countable which is a contradiction.

  Therefore there exists an infinite subsequence $H_{n_1}\supsetneq H_{n_2}\supsetneq\ldots\supsetneq H_{n_k}\supsetneq\ldots$. Choose a sequence of homomorphisms $f_k\in H_{n_k}\setminus H_{n_{k+1}}$. Then every sequence $s=(s_k)_{k<\omega}$ in $R^\omega$ defines a homomorphism $f_s=\sum_k f_ks_k$, since the sum is finite on every $a_n$. If $k$ is the least index such that $s_k\neq 0$, then $(a_i)f_k\neq 0$ for some $n_k<i\leq n_{k+1}$, hence $(a_i)f_s=(a_i)f_ks_k\neq 0$. Therefore $f_s\neq 0$ for a nonzero sequence $s$, and therefore we obtain a submodule $\prod_{k < \omega}R f_k\subseteq H$.
\end{proof}

\begin{corollary}
  If $G$ is an almost full embedding of a category of graphs into a category of torsion-free $R$-modules and $\Hom_R(X,Y)$ is uncountable, then either $GX$ or $GY$ has to be uncountable.
\end{corollary}

\begin{proof} If  $GX$ and $GY$ are both countable, torsion-free $R$-modules, then $\Hom_R(GX,GY)$ contains by Lemma \ref{BS-group} a copy of $R^\omega$, hence it is not free. However,  $G$ is an almost full embedding (as in Theorem \ref{axThm}), and so $\Hom_R(GX,GY)$ must be free. This shows the corollary. \end{proof}

\section{Applications to particular classes of $R$-objects}\label{section-applications}

For several applications of Main Theorem~\ref{axThm} we will consider non-trivial commutative rings $R$ with a distinguished countable subset $\bS$. We require that $0\notin\bS$, $1\in\bS$ and $st\in\bS$ whenever $s,t\in\bS$.

An $R$-module $M$ is {\em $\bS$-reduced} if $\bigcap_{s\in\bS}Ms =0$ and {\em $\bS$-torsion-free} if $ms=0, m\in M, s\in \bS$ implies $m=0$. We assume that $R$ is an $\bS$-ring, that means it is both $\bS$-reduced and $\bS$-torsion-free. The condition that the module $M$ is $\bS$-reduced is equivalent to the fact that the $\bS$-topology (generated from a basis of neighbourhoods $\{sR\mid s\in \bS\}$ of $0$ on $R$) is Hausdorff. We say that $R$ is cotorsion-free (with respect to $\bS$) if the $\bS$-completion $\Rhat$ of $R$, with respect to the $\bS$-topology, satisfies $\Hom_R(\Rhat,R)=0$. If $M$ is an $R$-module of size $<\Cont$, then $M$ is $\bS$-cotorsion-free if and only if it is $\bS$-torsion-free and $\bS$-reduced, as shown in \cite[Corollary 1.25, Vol. 1]{GT}. This definition extends naturally to any $R$-algebra $A$, where cotorsion-freeness is defined on $A_R$ -- the $R$-module structure of $A$.

We want to apply our Main Theorem~\ref{axThm} by using various realization theorems of algebras as endomorphism algebras of $R$-objects.

\subsection{Application of Corner's realization theorem}\label{corn}

Our first application of Main Theorem~\ref{axThm} will be based on an extension of Corner's celebrated theorem representing all countable rings with torsion-free, reduced additive structure as endomorphism rings of abelian groups, see \cite[Theorem A]{corner}.

Our Main Theorem~\ref{axThm} requires only that the $R$-algebra $A$ which should produce $M$ is free (i.e. free as an $R$-module), but we obtain in addition $\size{A}=\size{M}$. In the case of abelian groups of finite rank, the existence of such a pair $(A,M)$ with $\rk A = \rk M$ and $\End_R M=A$ was (after Corner's theorem) first a conjecture by Zassenhaus, then answered positively in Zassenhaus \cite{Zas} and extended by Butler \cite{But} in the last century. One of us noted in \cite{P} that in case of abelian groups of infinite rank Corner's original proof can be slightly modified to provide the existence of a pair $(A,M)$, of size $2^{\aleph_0}$, with $A$ a free ring and $M$ an abelian group as above. Since we are now working over more general rings $R$ and free $R$-algebras $A$, we must extend \cite[Corollary 20.2, p. 534, Vol. 2]{GT}.

\begin{theorem} \label{genCo} Let $R$ be an $\bS$-ring of cardinality $<\Cont$ and $A$ a free $R$-algebra of rank $\le \Cont$ which acts faithfully (on the right) of a free $R$-module $C$ of rank $ \l\le \Cont$. Then there exists a family $\{G_X\mid  X\subseteq
\l\}$ of $\bS$-reduced and $\bS$-torsion-free $R$-modules such
that  $C\subseteq G_X\subseteq G_{X'} \subseteq \Chat$ for
$X\subseteq X'\subseteq \l$, $\size{G_X}= \size{C}$ and, for any $X,X'\subseteq \l$,
\begin{equation*}
\Hom_R(G_X,G_{X'}) \cong
\begin{cases}
\ A & \text{if $\ \  X \subseteq X'$}\\
\ \ 0, &  \text{if $\ \  X \not\subseteq X'$.}
\end{cases}
\end{equation*}
\end{theorem}

We leave it as an exercise to derive a similar theorem, strengthening the isomorphism to a topological isomorphism similar to \cite[p. 532, Theorem 20.1, Vol. 2]{GT}. Also note that  the theorem applies for $C=A$, and if $C$ is countable, then all $G_X$ are countable.

\Pf Let  $C=\bigoplus_{i<\l}e_iR$ (and note that Theorem~\ref{genCo} follows immediately from \cite[p. 532, Theorem 20.1, Vol. 2]{GT} if $\l < \Cont$).  Since $\size{R} <\Cont$ we obtain from \cite[p. 17, Theorem 1.21, Vol.~1]{GT} the existence of $\Cont$ pure elements in $\Rhat$, which are algebraically independent over $R$. We can essentially follow the arguments of the proof of \cite[p. 532, Theorem 20.1, Vol. 2]{GT}, but must twice switch between $R$ and $A$. By the cardinal assumptions  we can enumerate a subfamily of algebraically independent elements $w_x, w_c$ ($x\in\l, c\in C$) without repetition by the indexing set $C\, \dot\cup \l$.

For $X\subseteq \l$ let
$$H_X= C + \sum_{c\in C}cAw_c + \sum_{x\in X}Cw_x \text{ and } C\subseteq G_X= (H_X)_*\subseteq \Chat$$
where $G_X$ is the $\bS$-purification of $H_X$, which can be expressed as $G_X=\Chat \cap \bS^{-1}H_X$.

If $X\subseteq X'\subseteq \l$, then $G_X\subseteq G_{X'}$ are $A$-submodules and naturally $$A\subseteq \Hom_R(G_X,G_{X'}).$$ The key step for the proof of the theorem is the following claim.

$$ \text{If } c\in C \text{ and } cA \text{ is  an
$\bS$-pure submodule of } C,$$\vspace*{-1cm}
\begin{eqnarray}\label{equcoeff}\text{then } \ G_X\cap G_Xw_c =
cAw_c.\end{eqnarray} We only have to show that $G_X\cap G_Xw_c
\subseteq cAw_c$. So let $x\in G_X\cap G_Xw_c$, and we may also assume that $x\in H_X\cap H_Xw_c$. Then we can write
\begin{equation}\label{equation-x-intersection}
x = c'' + \sum\limits_{d\in C} dw_da_d + \sum\limits_{x\in X}
c_xw_x = c'w_c + \sum\limits_{d\in C} dw_da'_dw_c +
\sum\limits_{x\in X} c'_xw_xw_c
\end{equation}
for suitable elements
$c, c',c'',c_x, c'_x\in C$ and $a_d,a'_d\in A$.

If $c=\sum_{i<\l}e_i c_i\in \widehat{\bigoplus_{i<\l}e_i R}=\Chat$ with $c_i\in \Rhat$, then $(cw)_i =c_iw$ for $w\in \Rhat$, because $\Chat$ is naturally an $\Rhat$-module. Now we can restrict Equation~(\ref{equation-x-intersection}) to the components at $e_i\Rhat$ and get the following equalities for all $i<\l$.
$$0 = c''_i + ((ca_c)_i - c'_i) w_c + \sum\limits_{c\ne d\in C} (da_d)_iw_d -
\sum\limits_{d\in C} (da'_d)_i(w_dw_c) + \sum\limits_{x\in X} (c_x)_iw_x -
\sum\limits_{x\in X} (c'_x)_i (w_xw_c),$$
with coefficients in $R$ and $w_c,w_d, \dots $ algebraically independent over $R$. Thus $(ca_c)_i = c'_i$ for all $i <\l$ and all other coefficients are $0$. These equalities give $c a_c = c'$ and (\ref{equation-x-intersection}) implies  $x = ca_c w_c\in cAw_c \cap \Chat$ and so $x\in cAw_c$. Hence (\ref{equcoeff})
follows.\smallskip

Now we can continue the proof of Theorem 20.1 in \cite{GT}, with some simplification, because we work with the discrete topology.

Let $X,X'\subseteq \l$ and let $\va: G_X\arr G_{X'}$ be an
$R$-homomorphism. We derive the next claims (as in \cite{GT}).
\begin{eqnarray}\label{v-invar} \text{If } c\in C \text{ and }
cA \text{ is  an $\bS$-pure submodule of }C \text{, then }
cA\va\subseteq cA.\end{eqnarray}

In particular, we now have that, for each $i<\l$, there exists
$a_i\in A$,  such that $e_i\va = e_ia_i$  If $i,j <\l$ are distinct, then
$(e_i+e_j)A$ is a direct summand of $A$. Thus, by (\ref{v-invar}),
there exists $a' \in A$  such that
$$e_i a_i + e_j a_j = (e_i +e_j)\va = (e_i + e_j) a' = e_i a' + e_j
a'$$ and hence there is $a\in A$ such that $$ e_i\va = e_i a \ \text{ for all }
i<\l.$$
Since $e_i$'s generate $C$ we derive that $\va=a$, which answers half of the theorem.

It remains to consider $X\nsubseteq X'$.  We choose $x\in X\setminus X'$. Then $Aw_x\cap
H_{X'} = 0$ which is shown similarly to (\ref{equcoeff}). However,
$cw_x\va = caw_x\in Aw_x\cap H_{X'}$ for every $c\in A$, and hence
$ca = 0$. Consequently, $a=0$ and thus $\va = 0$. \Fp

From Theorem~\ref{genCo} we derive an obvious
\begin{corollary} \label{theorem-corner}
If $A$ is an $R$-algebra with free additive structure over an $\bS$-torsion-free and $\bS$-reduced ring $R$ such that $\size{R}<\Cont, \size{A}\le \Cont$,  then there exists an $R$-module $M=G_\emptyset$ such that:
  \begin{enumerate}
    \item  $A\cong\End_RM$.
    \item  $A\subseteq M\subseteq\widehat{A}$ as left $A$-modules.
    \item  $|M|=|A|$.
  \end{enumerate}
\end{corollary}

We are ready to apply Main Theorem~\ref{axThm}.

Axioms A1 -- A4 are clear for the category of $R$-modules. Axiom B1 corresponds to (i). Axiom B3 follows since the $\bS$-completion preserves monomorphisms. Axiom B2 is a consequence of (ii) and the following:

\begin{remark}\label{remark-a-purity}
  The inclusion $A\varphi\subseteq A\cap M\varphi$ is obvious.
  Since $A$ is a free $R$-module and we have $A\subseteq M\subseteq \widehat{A}$, any element $m\in M$ can be uniquely written as a sum $m=\sum_{\sigma\in I}k_\sigma\sigma$ where $I$ is a countable subset of the basis of $A$, $k_\sigma\in R$, and the sequence $k_\sigma$ converges to $0$ in the $\bS$-topology. If $\varphi\in\Gamma$ and $m\varphi\in A$, then we have
  $$\sum_{\sigma\in I}k_\sigma\sigma\varphi=m\varphi=\sum_\tau r_\tau\tau$$
  where the right sum is the standard representation of a member of $A$: the sum is finite, $r_\tau\in R$ and the $\tau$'s belong to the basis of $A$. By uniqueness the left sum can be written as
  $$\sum_\tau(\sum_{\sigma\in I_\tau}k_\sigma)\tau$$
  where $I_\tau=\{\sigma\in I\mid \sigma\varphi=\tau\}$. Each $I_\tau$ must be nonempty although the union of $I_\tau$'s may not cover all of $I$. Choosing representatives $\sigma_\tau\in I_\tau$ we obtain
  $$m\varphi=\sum_\tau r_\tau\sigma_\tau\varphi=(\sum_\tau r_\tau\sigma_\tau)\varphi$$
  which implies the inclusion $A\varphi\supseteq A\cap M\varphi$ and therefore it verifies the second part of Axiom B2.
\end{remark}

As a consequence we obtain

\begin{corollary}\label{corollary-result-corner}
  Let $R$ be any reduced commutative ring of cardinality less than the continuum and whose additive group is torsion-free. Then there exists an almost full embedding of the category of graphs into the category of $R$-modules.
\end{corollary}

\subsection{Application of Shelah's Black Box}

The following theorem was derived in \cite{CG} by application of Shelah's Black Box (stated and proved in the appendix of \cite{CG}, for a more recent proof see also \cite[ The General Black Box 19.23, p. 508 and Theorem 20.42, p. 565, Vol. 2]{GT}).

\begin{theorem}\label{cotfreethm}
Let $A$ be an $R$-algebra, $|R|<\kappa$, with $A_R$
$\bS$-cotorsion-free. Moreover, suppose that $|A|\leq \k$ and $\lo
=\l^\k$. Then there are $\bS$-cotorsion-free $R$-submodules $G_X$
of $G=G_{\lo}$ of cardinality $\lo$ for all $X\subseteq \lo$ such
that the following holds.
$$
\Hom_R (G_X, G_{X'}) \cong \left\{
\begin{array}{rr}
A, \quad \text{\  if   }\quad  X \subseteq X'\\
0, \quad \text{ \ if   }\quad  X \not\subseteq X'
\end{array} \right.
$$
\end{theorem}

Next we apply our Main Theorem~\ref{axThm}, suppose that $\size{R}<\kappa$ and choose for the full subcategory $\Gamma$ all isomorphism classes of graphs of size $<\k$ and let $A$ be the free $R$-algebra as in Section~\ref{Axi}. Assuming that $R$ is cotorsion-free, it is immediate that also $A$ is cotorsion-free of cardinality $<\kappa$. For any cardinal $\l$ with $\lo
=\l^\k$ we find by Theorem~\ref{cotfreethm}  a cotorsion-free $R$-module $M=G_\emptyset$ of size $\size{M}=\lo$. From the construction it is also immediate that $M$ is sandwiched as $F\subseteq M\subseteq_*\Fhat$, where $M\subseteq_*\Fhat$ denotes a pure subgroup and  $F=\bigoplus_{\alpha<\l}e_\alpha A$ is a free $A$-module, thus a free $R$-module, and $\Fhat$ is the $\bS$-completion of $F$ such that $\End_R M=A$. Now we assume the mild condition that $\kappa^{\aleph_0}=\kappa$ (mainly saying that $\kappa \ge \Cont$) and check the axiomatic assumption of Theorem~\ref{axThm}.
Axiom B1 is clear. Since the inclusion $A\subseteq\widehat{A}$ is an $A$-retract of $F\subseteq\widehat{F}$ an argument as in Remark~\ref{remark-a-purity} implies Axiom B2.
Axiom B3 needs an easy consideration. By definition of $M$ it follows that $\id_C\in A$ for any $C\in \Gamma$, and if $\va:C\arr D$ is a monomorphism, then $\va=a\in A$ for some $a\in A$ acts by scalar multiplication on $M$. This is uniquely defined on $F$ and extends canonically to $M$. Now it is clear that $a\restr M\id_C$ is also a monomorphism. Thus Theorem~\ref{axThm} applies and we obtain the following

\begin{corollary}\label{corollary-result-black-box} Let $R$ be a cotorsion-free $\bS$-ring of size $<\kappa$ and $\kappa =\kappa^{\aleph_0}$. Then there is an almost full embedding of the category of graphs into the category of $R$-modules.
\end{corollary}

It is interesting to note that the above $R$-module $M$ over a countable, principal ideal domain is also $\aleph_1$-free, i.e. all its countably generated submodules are free. (This is immediate from the construction of $M$ using a support argument and Pontryagin's theorem, that countably generated $R$-modules are free if all submodules of finite rank are free, see \cite{Fu}.) Thus the functor obtained in Corollary~\ref{corollary-result-black-box} takes values in $\aleph_1$-free $R$-modules.

The notion $\ale$-freeness can be generalized to more general algebras (circumventing also Pontryagin's theorem).

An $A$-module $M$ is
$\k$-free if there is a family $\CC$  of $\bS$-pure $A$-submodules  of
$M$ satisfying the following conditions.
\begin{enumerate}
\item [(i)] Every element of $\CC$ is a $<\k$ generated free $A$-submodule of
$M$.
\item [(ii)]  Every subset of $M$ of cardinality $<\k$ is contained in an
element of $\CC$.
\item [(iii)]  $\CC$ is closed under unions of well-ordered chains of length
$<\k$.
\end{enumerate}
We say that $\CC$ is $<\kappa$-closed. This definition applies for regular cardinals, in particular for
$\k=\aleph_n$, which is the case we are interested in. Exploiting the $\ale$-freeness from the Black Box construction, just mentioned, after some more lengthy calculation in \cite{GHS} we obtain the following stronger result which can be applied for the embedding of graphs.

\begin{theorem}\label{thethm} If $R$ is a  cotorsion-free $\bS$-ring
and $A$ an $R$-algebra with free $R$-module $A_R$, $\size{A} <\mu, k <\omega$
and $\l= \beth^+_k(\mu)$, then we can construct an $\aleph_k$-free
$A$-module $G$ of cardinality $\l$ with $R$-endomorphism algebra
$\End_RG =A$.
\end{theorem}

Here $\l$ is obtained inductively. Let $\size{A} < \mu= \mu^{\size{A}}$, put $\beth^+_0(\mu)=\mu$
and $\beth^+_{n+1}(\mu) = (2^{\beth^+_n(\mu)})^+$  which is the
successor cardinal of the power set of $\beth_n^+(\mu)$. In a recent paper Herden \cite{herden} is able to replace $\beth^+_k(\mu)$ by $\beth_k(\mu)$, where inductively $\beth_0(\mu) = \mu$ and $\beth_{n+1}(\mu) = 2^{\beth_n(\mu)}$.

Thus we obtain the following stronger embedding theorem.

\begin{corollary}\label{corollary-stronger-embedding}  If $k$ is any natural number, $R$ is a cotorsion-free $\bS$-ring, then  there is an almost full embedding of the category of graphs into the category of $\aleph_k$-free
$R$-modules.\end{corollary}

\subsection{Application of Shelah's Elevator}
In this section we will first show the existence of an almost full embedding of graphs into the category of $R_4$-modules, where $R$ is any ring, possibly the field $\mathbb{F}_2$ with $2$ elements. Recall the following well-known definition from representation theory of algebras, see e.g. \cite[p. 603, Theorem 22.8, Vol. 2]{GT}.

\begin{definition}\label{definition-r-4-mod}$R_4$-{\bf Mod} denotes
the category of $R$--modules $M$ with $4$ distinguished submodules
${\bf M} = (M, M^k \ :k < 4)$. We will also
write ${\bf M} = (M, M^0,M^1,M^2,M^3)$. Call $\bf M$ as above an
$R_4$--module. Also recall that the homomorphisms between two $R_4$--modules ${\bf
M}, {\bf M'}$ must respect their distinguished submodules $M^k,
M'^k$, thus we define
\begin{align*}
\bHom_R({\bf M},{\bf M'})&={\Hom_R}\big(M,M'; M^k, M'^\k \ | \  k < 4\big) \\
& = \big\{\s \in \Hom_R(M,M'): M^k\s\subseteq M'^k \text{ for all }
k  < 4\big\}.
\end{align*}
\end{definition}

We quickly note, that $4$ in the next results is also minimal. E.g. $R_3$-modules over a field $R$ are of finite representation type, hence infinite dimensional $R_3$-vector spaces will decompose, as follows from \cite{RT,Si2}

Using the Shelah Elevator (see \cite[p. 515, Theorem 19.29, Vol. 2]{GT}) the following theorem  was shown in \cite{GM} (see also \cite[p. 603, Corollary 22.10, Vol. 2]{GT}).

\begin{theorem}\label{closedsub} Let $R\ne 0$ be a commutative ring,
$\l$ any infinite cardinal, and $A$ an $R$--algebra which is
generated by no more than $\l$ elements. Embed $A$ in $\End_RM$ by
scalar multiplication, where $M =\bigoplus_\l A$. Then there exists
an $R_4$--module ${\bf M} =
(M,M^0,M^1,M^2,M^3)$ with $A = \bEnd_R \bf M$, where ${\bEnd_R \bf M}
= \{\va \in \End_RM: M^k\va \subseteq M^k \text{ for } k <4\}$. All modules $M^k, M/M^k$ are $A$-free of rank $\l$.
\end{theorem}

The categories of $R_4$-modules over any commutative ring $R\ne 0$ satisfy the axioms from Section~\ref{Axi} and are important examples of $\rc$ categories , where B1, B2 and B3 are a consequence of Theorem~\ref{closedsub}. Thus Theorem~\ref{axThm} applies and we get the following

\begin{corollary}\label{corollary-r4} If $R\ne 0$ is any ring, there is an almost full embedding of the category of graphs into the category of $R_4$-modules. This also holds, when $R=\{0,1\}$.
\end{corollary}

This result also gives an immediate almost full embedding of graphs into the category of $R$-modules, if we assume that $R$ has at least $4$ comaximal primes $p$ with $\bigcap_{n<\omega} p^nR=0$. In this case we obtain a realization theorem for $R$-algebras as Theorem~\ref{cotfreethm} (replacing cotorsion-free by being $p$-reduced for four primes), see \cite[p. 639, Corollary 23.6, Vol. 2]{GT}. Thus a corresponding embedding corollary will hold.

\begin{corollary}\label{4primes} Let $R$ be a commutative ring with $4$ comaximal primes $p$ such that $\bigcap_{n<\omega} p^nR=0$. Then there is an almost full embedding of the category of graphs into the category of $R$-modules.\end{corollary}

\subsection{An Appendix on Absolute Constructions}
The following (large) strongly inaccessible cardinal is closely related to absolute properties, see also \cite{J} for the definition.

\begin{definition} \label{erdo}Let $\k(\omega)$ denote the first $\omega$-{\sl Erd\H{o}s cardinal}.
This is defined as the  smallest cardinal $\k$ such
 that $\k \arr (\omega)^ {<\omega}$, i.e. for every
function $f$  from the finite subsets of $\k$ to 2 there exist an
infinite subset $X \subset \k$ and a function $g: \omega \arr 2$ such
that $f(Y) = g(|Y|)$ for all finite subsets $Y$ of $X$.
\end{definition}

Shelah in 1982 characterized this cardinal showing the existence of certain tree-embeddings if and only their cardinalities are bounded by $\k(\omega)$, we refer to \cite{S} for the proofs. The key of his proof is a construction of absolute objects which do not depend on generic extension of the given universe of set theory. This was exploited, using \cite{S} and \cite{trnkova-book}, in \cite{DGP} and shifted to graphs, adding the absoluteness condition to the known result on graphs.
 \begin{proposition} If $M$ is a monoid, then there exists a graph $G$ with an isomorphism $\End G\cong M$ between the endomorphism monoid $\End(G)$ and $M$ which holds absolutely, so is remains an isomorphism in any in any  generic extension of the given universe of set theory if and only if $\size{M},\size{G} < \k(\omega)$.\end{proposition}

Thus it is reasonable for problems concerning absolute properties of graphs to restrict to objects of size $< \k(\omega)$. Fortunately we have a parallel result on
absolute $R_4$-modules derived from \cite{GS2,FuG}; see details in \cite[Chapter 22, Vol. 2]{GT}, which matches to graphs of size $< \k(\omega)$.

\begin{corollary} \label{abs4un}
Let $\l<\k(\omega)$ be any infinite cardinal and $A$ any faithful algebra over a commutative ring $R \ne 0$,
such that $A$ has fewer than $\l$ generators. Then there exists a
family of free right $A_4$-modules
$${\bM}_U  = (M_U,  M_U^0,M_U^1,
M_U^2, M_U^3) \qquad (U\subseteq \l),$$ where $M,
M_U^j, M/M_U^j$ are free $A$-modules of rank $\l$ for all $0\le j\le 3$,
such that
$$\bHom_R (\bM_U, \bM_V) = \left\{
\begin{array}{rr}
A,  \text{  if   }\quad  U \subseteq V\\
0, \text{  if   }\quad  U \not\subseteq V
\end{array} \right.
$$
holds in any generic extension of the universe.
\end{corollary}

By the above it is now clear that the absolute $R_4$-modules of size $<\k(\omega)$ satisfy our axioms from Section~\ref{Axi}. Hence, combining our knowledge on absolute graphs and absolute $R_4$-modules we derive a final embedding theorem for graphs of size $<\k(\omega)$. We say that a graph is absolute, if its endomorphism monoid remains the same in any generic extension of the universe. Similarly we say that an $R_4$-module $\bM$ is absolute, if $\bEnd_R\bM$  remains the same in any generic extension of the universe.

\begin{corollary} Let $R\ne 0$ be any ring of size $<\k(\omega)$. Then there is an almost full embedding of the category of absolute graphs $<\k(\omega)$ into the category of absolute $R_4$-modules.\end{corollary}

\noindent Address of the authors:

\vspace{7pt}

\noindent
R\"udiger G\"obel \\ Fakult\"at f\"ur Mathematik,\\
Universit\"at Duisburg-Essen,\\ Campus Essen, 45117 Essen, Germany \\
{\small e-mail: ruediger.goebel@uni-due.de}

\vspace{7pt}

\noindent
 Adam J. Prze\'zdziecki\\
 The Faculty of Applied Informatics and Mathematics, \\
 Warsaw University of Life Sciences -- SGGW, \\
 ul. Nowoursynowska 159, 02-776 Warszawa, Poland \\
 {\small e-mail: adam\_przezdziecki@sggw.pl}

\end{document}